\documentclass[reqno]{amsart}
\usepackage[T1]{fontenc}
\usepackage{dsfont}
\usepackage{mathrsfs}
\usepackage[colorlinks]{hyperref}
\usepackage{xcolor}
\usepackage[a4paper,asymmetric]{geometry}
\usepackage{mathscinet}
\usepackage{latexsym}
\usepackage{amsthm}
\usepackage{amssymb}
\usepackage{amsfonts}
\usepackage{amsmath}
\usepackage{longtable}
\usepackage{graphicx}
\usepackage{multirow}
\usepackage{multicol}

\newtheorem{theorem}{Theorem}[section]
\newtheorem{thm}[theorem]{Theorem}

\newtheorem{lem}[theorem]{Lemma}
\newtheorem{remark}[theorem]{Remark}
\newtheorem{proposition}[theorem]{Proposition}
\newtheorem{prop}[theorem]{Proposition}
\newtheorem{corollary}[theorem]{Corollary}

\theoremstyle{definition}

\newtheorem{defn}[theorem]{Definition}

\theoremstyle{remark}
\numberwithin{equation}{section}

 \DeclareMathAlphabet{\mathpzc}{OT1}{pzc}{m}{it}
 \DeclareMathAlphabet{\mathsfsl}{OT1}{cmss}{m}{sl}

  \newcommand{\FH}{\mathfrak{H}}

\newcommand{\dif}{\mathrm{d}}

\newcommand{\abs}[1]{\left\vert#1\right\vert}
\newcommand{\set}[1]{\left\{#1\right\}}

\newcommand{\norm}[1]{\left\Vert#1\right\Vert}

\newcommand{\E}{\mathbb{E}}

 \newcommand{\Rnum}{\mathbb{R}}

 \newcommand{\innp}[1]{\langle {#1}\rangle}
\newcommand{\Be}{\begin{equation}}
\newcommand{\Ee}{\end{equation}}
\newcommand{\Bs}{\begin{split}}
\newcommand{\Es}{\end{split}}
\newcommand{\Bes}{\begin{equation*}}
\newcommand{\Ees}{\end{equation*}}
\newcommand{\BT}{\begin{thm}}
\newcommand{\ET}{\end{thm}}
\newcommand{\Bp}{\begin{proof}}
\newcommand{\Ep}{\end{proof}}
\newcommand{\BL}{\begin{lem}}
\newcommand{\EL}{\end{lem}}
\newcommand{\BP}{\begin{proposition}}
\newcommand{\EP}{\end{proposition}}
\newcommand{\BC}{\begin{corollary}}
\newcommand{\EC}{\end{corollary}}
\newcommand{\BR}{\begin{remark}}
\newcommand{\ER}{\end{remark}}
\newcommand{\BD}{\begin{defn}}
\newcommand{\ED}{\end{defn}}
\newcommand{\BI}{\begin{itemize}}
\newcommand{\EI}{\end{itemize}}

\allowdisplaybreaks

\begin{document}
\title[Berry-Ess\'{e}en bound of fBm-OU processes]{Berry-Ess\'{e}en bound for  the Parameter Estimation  of Fractional Ornstein-Uhlenbeck Processes with the Hurst Parameter $H\in (0,\frac12)$
}
\author[Y. Chen]{Yong CHEN}
 \address{College of Mathematics and Information Science, Jiangxi Normal University, Nanchang, 330022, Jiangxi, China}
\email{zhishi@pku.org.cn}
 \author[Y. Li]{Ying LI}
 \address{School of Mathematics and Computional Science, Xiangtan University, Xiangtan, 411105, Hunan, China }
 \email{liying@xtu.edu.cn}
\begin{abstract}
For an Ornstein-Uhlenbeck process driven by a fractional Brownian motion with Hurst parameter $H\in (0,\frac12)$, one shows the Berry-Ess\'{e}en bound of the least squares estimator of the drift parameter. Thus, a  problem left in Chen, Kuang, and Li 
2018 is solved, where the Berry-Ess\'{e}en bound
of the least squares estimator is proved for $H\in [\frac12, \frac34]$.  
A new ingredient is a corollary of the inner product's  representation of the Hilbert space associated with the fractional Brownian motion given by Jolis 2007. An approach based on Malliavin calculus given by Kim and Park 2017b is used.  Several computations are cited from  Hu, Nualart, and Zhou 2019.
\end{abstract}
\maketitle

\section{ Introduction}\label{sec 03}
 The statistical aspects of the following 1-dimensional Ornstein-Uhlenbeck process has been intensively studied by some authors recently.
\begin{equation}\label{fOU}
\mathrm{d} X_t= -\theta X_t\mathrm{d} t+\mathrm{d}B^{H}_t,\quad X_0=0,\quad 0\le t\le T,
\end{equation} where $B^{H}_t$ be a 1-dimensional fractional Brownian motion with Hurst parameter $H\in(0,\,1)$. Suppose that $H$ is fixed and known, then there are several types of estimators to the drift coefficient.  

Based on the continuous observation, the following maximum likelihood estimator is proposed:
\begin{align*}
\hat{\theta}_{MLE}&=-\Big\{ \int_{0}^T Q^2(s) \dif w_s^H \Big\} ^{-1}\int_0^TQ(s)\dif Z(s),
\end{align*} where 
\begin{align*}
Q(t)= \frac{\dif }{\dif w_t^H} \int_0^t k_{H}(t,s)X_s\dif s,\qquad Z(t)= \int_0^t k_{H}(t,s) \dif X_s,\\
k_{H}(t,s)=\kappa_H^{-1} s^{\frac12 -H}(t-s)^{\frac12 -H},\qquad w_t^H=\lambda_{H}^{-1} t^{2-2H}
 \end{align*}with constants $\kappa_H,\, \lambda_H$ depending on $H$. Please refer to Kleptsyna  and  Le Breton  2002 and Tudor and Viens 2007,  where the almost sure convergence of both the MLE and a version of the MLE using discrete observations for all $H\in (0, 1)$ is shown. Later on, the central limit theorem of $\hat{\theta}_{MLE}$ is shown in Bercu,  Coutin, and  Savy  2011  and Brouste and  Kleptsyna 2010.

The least squares estimator of the drift coefficient is given by a ratio of two Gaussian functionals (Hu and Nualart 2010):
\begin{equation}\label{hattheta}
\hat{\theta}_T=-\frac{\int_0^T X_t\mathrm{d}X_t}{\int_0^T X_t^2\mathrm{d} t}=\theta-\frac{\int_0^T X_t\mathrm{d}B^{H}_t}{\int_0^T X_t^2\mathrm{d} t},
\end{equation}where $\dif B_t^H$ denotes the  divergence integral or the extended divergence integral (see Cheridito and Nualart 2005). In case of $H\in (0,\,\frac34]$, the strong consistency and asymptotic normality of the estimator $\hat{\theta}_T$ are shown in Hu and Nualart 2010 and Hu, Nualart and Zhou 2019.  It is worth noting that several crucial computations in this paper come from that given in Hu, Nualart and Zhou 2019. 

It is well known that $\hat{\theta}_T$ cannot be computed from the path of $X$ since the translation between divergence and Young integrals relies on the parameter $\theta$ that is being estimated. This makes many authors study other more practical and difficult parameter estimate based on discrete observations (e.g. Barboza and Viens  2017; Es-Sebaiy  and Viens 2016; Sottinen and Viitasaari 2018). For example, in Es-Sebaiy 2013, a discrete time least squares estimator  
\begin{align*}
\hat{\theta}_{n}:=-\frac{\sum_{i=1}^n X_{t_{i-1}}(X_{t_i}-X_{t_{i-1}})}{\Delta_n\sum_{i=1}^nX_{t_i}^2},
\end{align*}where $t_i=i\Delta_n$, is proposed and an upper Berry-Ess\'{e}en-type bound in the Kolmogorov distance for $\hat{\theta}_{n}$ is  shown  when $\Delta_n\to 0$ and $n\to \infty$.  Moreover,  the so-called ``polynomial variation'' estimator is proposed and an upper Berry-Ess\'{e}en-type bound in the Wasserstein distance is shown in Onsy,  Es-Sebaiy, and Viens 2017.
 It is also found out that to discretize the continuous-time estimator will lost the estimator's interpretation as a least square optimizer (Onsy,  Es-Sebaiy, and Viens 2017). 
 
But it is still meaningful to study the property of $\hat{\theta}_T $ because it is a first step to understand the problem of parameter estimate for the 1-dimensional  fractional Ornstein-Uhlenbeck process (\ref{fOU})  such as its Berry-Ess\'{e}en behavior. Recently, by using an approach based on Malliavin calculus given by Kim and Park 2017b, it is shown in Chen, Kuang, and Li 2018  that as $T\to \infty$, when $H\in [\frac12,\frac34)$,  the Berry-Ess\'{e}en bound of $\sqrt{T}(\hat{\theta}_T-\theta)$ in the Kolmogorov distance is $\frac{1}{T^{\beta}}$, where $\beta=\frac12,\,\frac38-,\, 3-4H$ for $H\in [\frac12,\frac58),\,H=\frac58,\,H\in (\frac58,\frac34)$ respectively;  when $H=\frac34$, the Berry-Ess\'{e}en bound of $\sqrt{\frac{T}{\log T}}(\hat{\theta}_T-\theta)$ in the Kolmogorov distance is $\frac{1}{\log T}$. 
In fact, when $H=\frac12$, the Berry-Ess\'{e}en bound of $\sqrt{T}(\hat{\theta}_T-\theta)$ in the Kolmogorov distance is well known, please refer to Bishwal  2000, 2008,  and  the references therein. 

Since it involves much more complicated  method to calculate the inner product of the Hilbert space associated to the fractional Brownian motion in the case of $H\in (0,\frac12)$, 
 the Berry-Ess\'{e}en bound of $\sqrt{T}(\hat{\theta}_T-\theta)$ is still unknown for $H\in (0,\frac12)$. 
In this paper, we will give an affirmative answer to this question. The main result of the present paper is as follows. 
\begin{thm} \label{main thm}
Let $Z$ be a standard Gaussian random variable.
When $H\in (0, \,\frac12)$, there exists a constant $C_{\theta, H}$ such that when $T$ is large enough, 
\begin{equation}\label{hle 34}
\sup_{z\in \Rnum}\abs{P(\sqrt{\frac{T}{\theta \sigma^2_H}} (\hat{\theta}_T-\theta )\le z)-P(Z\le z)}\le\frac{ C_{\theta, H}}{ T^{(1-2H)\wedge \frac12}};
\end{equation}
where $\sigma^2_{H}$ is given in Hu, Nualart, and Zhou 2019 as follows:
\begin{equation}\label{sigmah}
\sigma^2_H= (4H-1) +   \frac{2 \Gamma(2-4H)\Gamma(4H)}{\Gamma(2H)\Gamma(1-2H)}.
\end{equation}
\end{thm}  Proof of Theorem~\ref{main thm} will be given in Section~\ref{sec prf}. The main idea to show Theorem~\ref{main thm} will be given in Section~\ref{prelim}. 

Theorem~\ref{main thm} implies that when $H\in (0,\,\frac14]$,  the Berry-Ess\'{e}en bound is $\frac{c}{\sqrt{T}}$. When $H\in (\frac14,\,\frac12)$, the Berry-Ess\'{e}en bound is $\frac{c}{{T^{1-2H}}}$. It is known that when $H=\frac12$, the optimal Berry-Ess\'{e}en bound is $\frac{c}{\sqrt{T}}$  (Kim and Park 2017a, 2017b).  Thus, it is reasonable to conjecture that when $H\in (\frac14,\,\frac12)$, a better bound should be $\frac{c}{\sqrt{T}}$ .
 This improving topic will be investigated in other works. In the remaining part of this paper, $c$ will be a generic positive constant whose values may differ from line to line.

\section{Preliminary}  \label{prelim}
The fractional Brownian motion (fBm) $B^H = \set{B^H_t , t \in[0,T]}$ with Hurst
parameter $H \in (0, 1)$ is a continuous centered Gaussian process, defined on a complete probability
space $(\Omega,\mathcal{F}, P )$, with covariance function given by
$$\E(B^H_t B^H_s)=R_H(t,s)=\frac12 \big(\abs{t}^{2H}+\abs{s}^{2H}-\abs{t-s}^{2H}\big).$$
Let $\mathcal{E}$ denote the space of all real valued step functions on $[0,T]$. The Hilbert space $\mathfrak{H}$ is defined
as the closure of $\mathcal{E}$ endowed with the inner product
\begin{align*}
\innp{\mathbf{1}_{[a,b)},\,\mathbf{1}_{[c,d)}}_{\FH}=\E\big(( B^H_b-B^H_a) ( B^H_d-B^H_c) \big).
\end{align*}
In the case $H\le \frac12$, this space is a space of functions,  and when $H>\frac12$, this space contains distributions that are not given
by functions, please refer to Jolis 2007, Pipiras and Taqqu 2000 and 2001.

The following proposition is an adaptation of Theorem 2.3 of Jolis 2007.
\begin{prop}
Denote $\mathcal{V}_{[0,T]}$ the set of bounded variation functions on $[a,b ]$. Then $\mathcal{V}_{[0,T]}$ is dense in $\FH$. 
Moreover, if $f,\, g\in \mathcal{V}_{[0,T]}$, one has that
\begin{align} 
\innp{f,g}_{\FH}&=\int_{[0,T]^2} R_H(t,s) \nu_f( \dif t)   \nu_{g}( \dif s), \label{innp fg30}\\ 
&=-\int_{[0,T]^2}  f(t) \frac{\partial R_H(t,s)}{\partial t} \dif t  \nu_{g}(\dif s).\label{innp fg3}
\end{align}
where $\nu_{g}$ is the restriction to $([0,T ], \mathcal{B}([0,T ]))$ of 
the Lebesgue-Stieljes signed measure associated with $g^0$ defined as
\begin{equation*}
g^0(x)=\left\{
      \begin{array}{ll}
 g(x), & \quad \text{if } x\in [0,T].\\
0, &\quad \text{otherwise }.     
 \end{array}
\right.
\end{equation*}
\end{prop}
\begin{proof}The first claim and the identity (\ref{innp fg30}) are cited from Theorem 2.3 of Jolis 2007. The  identity (\ref{innp fg3})  can be implied from the formula of integrations by parts. 
In fact,  for the step functions on $[0,T ]$ of the form
\begin{align*}
f=\sum_{j=0}^{N-1}f_j \mathbf{1}_{[t_j,t_{j+1})},
\end{align*}
where $\set{0= t_0 < t_1 < \cdots < t_N = T } $is a partition of $[0,T ]$ and $f_j\in\Rnum$. The corresponding signed measure is 
\begin{align*}
\nu_{f}=\sum_{j=1}^{N-1}( f_j-f_{j-1})\delta_{t_j}+f(0+)\delta_0-f(T-)\delta_T.
\end{align*} It is clear that the following formula of integrations by parts hold: for any $s\in [0,T]$,
\begin{align}\label{integ by parts 1}
-\int_{[0,T]}  f(t) \frac{\partial R_H(t,s)}{\partial t} \dif t =\int_{[0,T]} R_H(t,s) \nu_f( \dif t).
\end{align}
Next, given $f$ a right continuous monotone non-decreasing function on $[0,T]$ and a sequence partitions $\pi_n= \set{0 = t_0^n < t_1^n < \cdots < t_{k_n}^n = T } $ such that $\pi_n\subset \pi_{n+1}$ and $\abs{\pi_n}\to 0$ as $n \to \infty$, consider 
\begin{align*}
f_n=\sum_{j=0}^{{k_n}-1} f(t_j^n) \mathbf{1}_{[t_j^n,t_{j+1}^n)}.
\end{align*}
Hence, the sequence of signed measures $\nu_{f_n}$ converges weakly to $\nu_{f}$. Taking limit on both sides of (\ref{integ by parts 1}),  one has that
it is  still valid for right continuous monotone non-decreasing functions on $[0,T]$. Finally, it is well known that every function of bounded variation is the difference of two  monotone non-decreasing function and that the value of $f$ at its points of discontinuity are irrelevant for the purposes of determining  the Lebesgue-Stieltjes measure $\nu_f$ (Tao 2011). One has that (\ref{integ by parts 1}) is valid for any  $f\in \mathcal{V}_{[0,T]}$ and hence the  identity (\ref{innp fg3})  holds. 
\end{proof}

Especially, taking $g=h\cdot \mathbf{1}_{[a,b]}(\cdot)$ with $h$ a continuously differentiable function in (\ref{innp fg3}), one has a more explicit  inner product presentation using the distributional derivative. 
\begin{corollary}\label{cor new}
Denote by $\delta_a(\cdot)$ the Dirac delta function centered at a point $a$.  Let $g=h\cdot \mathbf{1}_{[0,T]}(\cdot)$ with $h$ a continuously differentiable function.  Then  one has 
\begin{align} \label{innpfg extend}
\norm{g}^2_{\FH}&=-\int_{[0,T]^2}  h(t)h'(s) \mathbf{1}_{[0,T]}(s)\frac{\partial R_H(t,s)}{\partial t} \dif t \dif s\nonumber\\ 
&\quad +\int_{[0,T]^2}  h(t)h(s)  \frac{\partial R_H(t,s)}{\partial t} [\delta_T(s)-\delta_0(s)] \dif t \dif s.
\end{align}
\end{corollary}
\begin{proof}
The Heaviside step function $\mathrm{H}(x)$ is defined as
\begin{equation*}
\mathrm{H}(x)=\left\{
      \begin{array}{ll}
 1, & \quad \text{if } x> 0,\\
0, &\quad \text{if } x<0.
     \end{array}
\right.
\end{equation*}
 The distributional derivative of the Heaviside step function is the Dirac delta function: $$\frac{\dif \mathrm{H}(x)}{\dif x}=\delta_0(x).$$
 Hence, one has that
  \begin{align}\label{distribution deriva}
 \frac{\dif }{\dif x} \mathbf{1}_{[0,T]}(x)= \frac{\dif }{\dif x} [\mathrm{H}(x)-\mathrm{H}(x-T)]=\delta_0(x)-\delta_T(x),
 \end{align}which implies that 
 \begin{align}\label{deriv dis 2}
  \nu_{g}(\dif  x)=\Big[h'(x) \mathbf{1}_{[a,b]}(x) +h(x)\big(\delta_0(x)-\delta_T(x)\big)\Big] \dif x.
 \end{align}
Substituting (\ref{deriv dis 2}) into (\ref{innp fg3}),  one has the desired inner product presentation (\ref{innpfg extend}). 
\end{proof}
\begin{remark}
If $f, g \in \FH$ and $g$ is a continuously differentiable function with compact support,  
it is proved in Hu, Jolis, and Tindel 2013 and Hu, Nualart and Zhou 2019 that
\begin{align}\label{innpfg}
\innp{f,\,g}_{\FH}=-\int_{[0,T]^2}  f(t)g'(s) \frac{\partial R_H(t,s)}{\partial t} \dif t \dif s.
\end{align}
The corollary implies that if $g'$ is interpreted as the distributional derivative, then the identity (\ref{innpfg}) maybe still holds for the functions such as $g=h\cdot \mathbf{1}_{[a,b]}(\cdot)$ with $h$ a continuously differentiable function. But one will not attempt to prove it here since the theory of fractional order Sobolev spaces is involved.
\end{remark}
It is well known that when $H\in (\frac12,\,1)$, for any
$f, g \in L^{\frac{1}{H}}([0, T ])$, if one extends $f$ and $g$ to be zero on $ \Rnum\cap [0,T]^{c}$, then $f,\,g \in \FH$ and (\ref{innpfg}) is equal to a simple identity
\begin{align}\label{innp fg hg12}
\innp{f,\,g}_{\FH}=H(2H-1)\int_{[0,T]^2} f(u)g(v)\abs{u-v}^{2H-2}\dif u\dif v.
\end{align} As one points out before, it is the difference between (\ref{innpfg extend}) and (\ref{innp fg hg12}) that leads to the case of $H\in (0,\frac12)$ much more complicated than the case of $H\in [\frac12, \frac34]$.

A Gaussian isonormal process associated with $\mathfrak{H}$ is given by Wiener integrals with
respect to a fBm for any deterministic kernel $f\in\mathfrak{H}$:
\begin{align*}
B^{H}(f)=\int_{0}^{\infty} f(s)\dif B^{H}_s.
\end{align*}

Let $H_n$ be the $n$-th Hermite polynomial. 
The closed linear subspace $\mathfrak{H}_n$ of $L^2(\Omega)$ generated by $\set{H_n(B^H(f)):\, f \in\mathfrak{H}, \norm{f}_{\mathfrak{H}}=1} $ is called the $n$-th Wiener-Ito chaos. The linear  isometric mapping $I_n:\, \mathfrak{H}^{\odot n}\to \mathfrak{H}_n$  given by $I_n(h^{\otimes n})= H_n(B^H(f))$ is called the $n$-th multiple Wiener-Ito integral. For any $f\in \mathfrak{H}^{\otimes n}$, define $I_n(f)=I_n(\tilde{f})$ where $\tilde{f}$ is the symmetrization of $f$.

Given $f\in \mathfrak{H}^{\odot p}$ and $g\in \mathfrak{H}^{\odot q}$ and $r=1,\cdots, p\wedge q$,  
$r$-th contraction between $f$ and $g$ is the element of $\mathfrak{H}^{\otimes (p+q-2r)}$ defined by
\begin{align*}
f\otimes_{r} g (t_1,\dots,t_{p+q-2r})&=\innp{f(t_1,\dots,t_r,\cdot),\,g(t_{r+1},\, \dots, t_{p+q-2r},\cdot)}_{\FH^{\otimes r}}.
\end{align*}

One will make use of the following estimate of the Kolmogrov distance between a nonlinear Gaussian functional and the standard normal (see Corollary 1 of Kim and Park 2017b).
\begin{thm}[Kim, Y. T., \& Park, H. S]\label{kp}
Suppose that $\varphi_T(t,s)$ and $\psi_T(t,s)$ are two functions on $\mathfrak{H}^{\otimes 2}$.
Let $b_T$ be a positive function of $\,T$ such that $I_2(\psi_T)+b_T>0$ a.s. If $\Psi_i(T)\to 0,\,i=1,2,3$ as $T\to \infty$, then there exists a constant $c$ such that for $T$ large enough, 
\begin{equation}\label{psi t}
\sup_{z\in \Rnum}\abs{P(\frac{I_2(\varphi_T)}{ I_2(\psi_T)+b_T}\le z)-P(Z\le z)}\le c\times \max_{i=1,2,3} \Psi_i(T),
\end{equation}
where 
\begin{align*}
\Psi_1(T)&=\frac{1}{b_T^2}\sqrt{\big[b^2_T-2\norm{\varphi_T}_{\mathfrak{H}^{\otimes 2}}^2\big]^2+8\norm{\varphi_T \otimes_1 \varphi_T}_{\mathfrak{H}^{\otimes 2}}^2},\\
\Psi_2(T)&=\frac{2}{b_T^2}\sqrt{2\norm{\varphi_T \otimes_1 \psi_T}_{\mathfrak{H}^{\otimes 2}}^2+\innp{\varphi_T,\,\psi_T}_{\mathfrak{H}^{\otimes 2}}^2},\\
\Psi_3(T)&=\frac{2}{b_T^2}\sqrt{ \norm{\psi_T}_{\mathfrak{H}^{\otimes 2}}^4+2\norm{\psi_T \otimes_1\psi_T}_{\mathfrak{H}^{\otimes 2}}^2}.
\end{align*}
\end{thm}

It follows from  Eq.(\ref{hattheta}) and the product formula of multiple integrals that 
\begin{equation}\label{ratio 1}
\sqrt{\frac{T}{\theta \sigma^2_H}} (\hat{\theta}_T-\theta )=\frac{I_2(f_T)}{ I_2(g_T)+b_T},
\end{equation}
where
\begin{align}
f_T(t,s)&=\frac{1}{2\sqrt{\theta\sigma^2_H T}}e^{-\theta \abs{t-s}}\mathbf{1}_{\set{0\le s,t\le T}},\\
g_T(t,s)&=\sqrt{\frac{\sigma^2_H}{\theta T}}f_T-\frac{1}{2\theta T}h_T,\label{gt ts}\\
h_T(t,s)&= e^{-\theta (T-t)-\theta (T-s)}\mathbf{1}_{\set{0\le s,t\le T}},\label{ht ts}\\
b_T&=\frac{1}{T}\int_0^T\, \norm{e^{-\theta (t-\cdot) }\mathrm{1}_{[0,t]}(\cdot)}^2 _{\mathfrak{H}}\dif t.\label{bt bt}
\end{align} The reader can also refer to Eq.(17)-(19) of Kim and Park 2017a for details.
By Theorem~\ref{kp} and the identity (\ref{ratio 1}), to obtain the Berry-Ess\'{e}en bound of $ \hat{\theta}_T$, one need to estimate the right hand side of (\ref{psi t}) which are several integrals. This is the main idea of the present paper and the previous paper Chen, Kuang, and Li 2018.
\section{Proof of Theorem \ref{main thm}} \label{sec prf}
%

One divides the estimate of the right hand side of (\ref{psi t}) into several lemmas.
The following estimate is cited from the inequality (3.17) of Hu, Nualart, and Zhou 2019.
\begin{lem}\label{zhou}
When $ H\in(0,\,\frac12)$, there exists a constant $C_{\theta,\,H}$ such that 
\begin{equation}\label{zhou ineq}
\norm{f_T\otimes_{1} f_T}_{\mathfrak{H}^{\otimes 2}}\le  \frac{C_{\theta,\,H}}{\sqrt{T}} .
\end{equation}
\end{lem} It is worth noting that to show the estimate (\ref{zhou ineq}), the Fourier transform is used to compute the inner product of the Hilbert space $\FH$ (Pipiras  and  Taqqu 2000):
\begin{align*}
\innp{f,\,g}_{\FH}=\frac{\Gamma(2H+1)\sin(\pi H)}{2\pi}\int_{\Rnum} \mathcal{F}f (\xi)\overline{\mathcal{F}g (\xi)} \abs{\xi}^{1-2H}\,\dif \xi.
\end{align*} Although the estimate (\ref{zhou ineq}) is crucial to the present paper,  one will not use this method  to compute the inner product any more in this paper.


\begin{lem} \label{bt exponential}
When $H\in (0,\,1)$,   the speed of convergence  $b_T\to H\Gamma(2H)\theta^{-2H}$ as $T\to \infty$ is at least $\frac{1}{T }$.
\end{lem}
\begin{proof} 

The symmetry and Corollary~\ref{cor new}  imply that 
\begin{align}
b_T &{ =-\frac{1  }{T}}\int_0^Te^{-2 \theta t}\dif t \int_{\Rnum^2}\frac{ \partial}{\partial v} \big[e^{\theta (u+v)}\mathbf{1}_{[0,t]^2}(u,\,v)\big] \frac{\partial R_H(u,v)}{\partial u} \dif u\dif v\nonumber \\
&=\frac{\theta H}{T}\int_0^Te^{-2 \theta t}\dif t\int_{[0,t]^2}e^{\theta (u+v)}  \big(  \mathrm{sgn}(u-v)\abs{u-v}^{2H-1} -\abs{u}^{2H-1}\big) \dif u\dif v\nonumber\\
& +\frac{ H}{T}\int_0^T e^{-2 \theta t}\dif t\int_{\Rnum^2}e^{\theta (u+v)}\mathbf{1}_{[0,t]}(u)(\delta_0(v)- \delta_t(v)) \big(  \mathrm{sgn}(u-v)\abs{u-v}^{2H-1} -\abs{u}^{2H-1}\big)\dif u\dif v\nonumber\\\
&=\frac{ H}{T}\int_0^Te^{-2 \theta t}\dif t\Big[ \big(-\theta \int_{[0,t]^2}e^{\theta (u+v)} {u}^{2H-1} \dif u\dif v+ \int_0^t e^{\theta  (u+t)}  u^{2H-1}\dif u\big)+ \int_0^t e^{\theta (u+t)} (t-u)^{2H-1} \dif u\Big]\nonumber\\
&=\frac{ H}{T}\int_0^Te^{-2 \theta t}\dif t\Big[   \int_0^t e^{\theta  u}  u^{2H-1}\dif u + \int_0^t e^{\theta (u+t)} (t-u)^{2H-1} \dif u\Big]\nonumber\\
&:=B_1+B_2,\label{bt expansion}
\end{align}where $B_1,\,B_2$ and their convergence speeds are given respectively as follows.
Integration by parts implies that there exists a constant $C_{\theta,H}$ such that 
\begin{align}
  0< B_1& =   \frac{ H}{T }\int_0^Te^{-2 \theta t}\dif t  \int_0^t e^{\theta  u}  u^{2H-1}\dif u \nonumber\\
&=  \frac{ H}{ {2 \theta T} }\Big[\frac{1}{e^{2\theta T} }\int_0^T e^{\theta  u}  u^{2H-1}\dif u  + \int_0^T e^{-\theta t} t^{2H-1}\dif t\Big]\nonumber\\
&\le \frac{C_{\theta,H}}{T}. \label{bt 1}
\end{align} 
Making change of variable $z=t-u$ and then integration by parts, one has that   
\begin{align*}
 B_2- H\Gamma(2H)\theta^{-2H}&= \frac{{H}}{T}\int_0^Te^{- \theta t}\dif t  \int_0^t e^{\theta u}  (t-u)^{2H-1} \dif u-H\Gamma(2H)\theta^{-2H}\\
&= \frac{{H}}{T}\int_0^T\dif t  \int_0^t e^{-\theta z}  z^{2H-1} \dif z-H\Gamma(2H)\theta^{-2H}\\
&=   \frac{{H}}{T}\Big[T \int_0^T e^{-\theta t}  t^{2H-1}\dif t   - \int_0^T e^{-\theta t}  t^{2H}\dif t\Big] -H\Gamma(2H)\theta^{-2H}\\
&=H \int_T^{\infty} e^{-\theta t}  t^{2H-1}\dif t -\frac{H}{T} \int_0^T e^{-\theta t}  t^{2H}\dif t.
\end{align*} Hence, there exists a constant $C'_{\theta,H}$ such that
\begin{align}
\abs{ B_2- H\Gamma(2H)\theta^{-2H}}&\le H \int_T^{\infty} e^{-\theta t}  t^{2H-1}\dif t +\frac{H}{T} \int_0^T e^{-\theta t}  t^{2H}\dif t\nonumber\\
&\le  \frac{C'_{\theta,H}}{T}.\label{bt 2}
\end{align}
Combining the limits (\ref{bt 1}) and (\ref{bt 2}) with the equality (\ref{bt expansion}), one has that the speed of convergence $b_T\to H\Gamma(2H)\theta^{-2H}$ is  at least $\frac{1}{T }$.
\end{proof}
\begin{remark}
In the case of $H\in [\frac12,\,\frac34)$, the same conclusion is shown in Chen, Kuang, and Li 2018. The proof in the present paper is suited to all $H\in (0,\,1)$.
\end{remark}
\begin{lem}\label{ht limit}
Let $h_T$ be given as in (\ref{ht ts}) and $H\in (0,\frac34)$. Then as $T\to \infty$,   
\begin{equation}
\frac{1}{\sqrt{T}} h_T \to 0,\quad \text{ in} \quad \mathfrak{H}^{\otimes 2}. 
\end{equation}
\end{lem}
\begin{proof} 
Without loss of generality, we can assume that $\theta=1$. Denote $\vec{t}=(t_1,\,t_2),\,\vec{s}=(s_1,\,s_2)$. The identity (\ref{innpfg extend}) implies that
\begin{align}
\frac{1}{T}\norm{h_T}_{\mathfrak{H}^{\otimes 2}}^2&=\frac{1}{T e^{4T}}\int_{\Rnum^4}\frac{ \partial^2}{\partial t_1 \partial s_2} \big[e^{t_1+s_1+t_2+s_2}\mathbf{1}_{[0,T]^4}(t_1,\,s_1,\,t_2,\,s_2)\big] \frac{\partial R_H(t_1,t_2)}{\partial t_2} \frac{\partial R_H(s_1,s_2)}{\partial s_1} \dif \vec{t} \dif \vec{s}\nonumber \\
&=\frac{1}{T e^{4T}}\int_{\Rnum^4} \mathbf{1}_{[0,T]^2}( s_1,\,t_2 ) e^{t_1+s_1+t_2+s_2} \frac{\partial R_H(t_1,t_2)}{\partial t_2} \frac{\partial R_H(s_1,s_2)}{\partial s_1}\nonumber \\
&\times  \Big[\mathbf{1}_{[0,T]^2}( t_1,\,s_2 ) +\mathbf{1}_{[0,T] }( t_1  ) (\delta_0(s_2)-\delta_T(s_2))\nonumber \\
&+\mathbf{1}_{[0,T] }( s_2  ) (\delta_0(t_1)-\delta_T(t_1))+( \delta_0(s_2)-\delta_T(s_2))(\delta_0(t_1)-\delta_T(t_1))\Big]\dif \vec{t} \dif \vec{s}\nonumber \\
&:=I_1+I_2+I_3+I_4.\label{ht exp}
\end{align}
By the symmetry and the L'Hospital's rule, one has that  
\begin{align*}
 \lim_{T\to \infty} I_1&= \lim_{T\to \infty} \frac{1}{T e^{4T}}\int_{[0,T]^4}  e^{t_1+s_1+t_2+s_2} \frac{\partial R_H(t_1,t_2)}{\partial t_2} \frac{\partial R_H(s_1,s_2)}{\partial s_1}\dif \vec{t} \dif \vec{s}\nonumber\\
&= \lim_{T\to \infty} \frac{2}{T e^{4T}}\Big[\int_{0\le t_2,s_1,s_2\le t_1\le T}  + \int_{0\le t_1,s_1,s_2\le t_2\le T}\Big]e^{t_1+s_1+t_2+s_2} \frac{\partial R_H(t_1,t_2)}{\partial t_2} \frac{\partial R_H(s_1,s_2)}{\partial s_1}\dif \vec{t} \dif \vec{s}\nonumber\\
&= \lim_{T\to \infty}  \frac{2H}{(1+4T)e^{3T}}\int_{[0,T]^3}  e^{t+s_1+s_2}(T^{2H-1}+t^{2H-1}) \frac{\partial R_H(s_1,s_2)}{\partial s_1}\dif {t} \dif {s}_1 \dif {s}_2\nonumber\\
&= \lim_{T\to \infty}  \frac{ H^2}{2T e^{3T}}\int_{[0,T]^3}  e^{t+s_1+s_2}(T^{2H-1}+t^{2H-1})s_1^{2H-1}\dif {t} \dif {s}_1 \dif {s}_2\nonumber\\
&= \lim_{T\to \infty}  \frac{ H^2}{2T e^{2T}}\int_{[0,T]^2}  e^{t+s_1 }(T^{2H-1}+t^{2H-1})s_1^{2H-1}\dif {t} \dif {s}_1 \nonumber \\
&= 0.
\end{align*}
In the same way, one has that   as $T\to \infty$, 
\begin{align*}
I_2=I_3&=-\frac{H}{T e^{3T}}\int_{[0,T]^3}  e^{ s_1+t+s_2}\big(t^{2H-1}+(T-t)^{2H-1}\big) \frac{\partial R_H(s_1,s_2)}{\partial s_1}\dif {t} \dif {s}_1 \dif {s}_2  \to 0.
\end{align*}
Finally,  one has that   as $T\to \infty$, 
\begin{align*}
I_4&=\frac{H^2}{T e^{2T}}\int_{[0,T]^2}  e^{ s +t }\big(t^{2H-1}+(T-t)^{2H-1}\big) \big(s^{2H-1}+(T-s)^{2H-1}\big)\dif {t} \dif {s} \nonumber\\
&=\frac{H^2}{T e^{2T}}\Big[\int_0^{T}  e^{ t }\big(t^{2H-1}+(T-t)^{2H-1}\big)  \dif {t}\Big]^2 \to 0.
\end{align*}
Combining the above three limits with the equality (\ref{ht exp}), one has that $\frac{1}{T}\norm{h_T}_{\mathfrak{H}^{\otimes 2}}^2\to 0$ as $T\to \infty$.
\end{proof}

Based on Lemma~\ref{ht limit},  one can obtain the following corollary whose proof is the same as Lemma~3.4 of Chen, Kuang, and Li 2018.
\begin{corollary}\label{gt gt}
Let $g_T$ and $\sigma^2_H$ be given as in (\ref{gt ts}) and  (\ref{sigmah}) respectively. Denote by $\delta_H=H^2\Gamma^2(2H)\sigma^2_H $. When $H\in (0,\,\frac12)$,  we have that as $T\to \infty$,
\begin{align*}
T\norm{g_T}_{\mathfrak{H}^{\otimes 2}}^2&\to \frac{\delta_H}{2\theta ^{1+4H}},\quad
T \innp{f_T,\,g_T}_{\mathfrak{H}^{\otimes 2}}^2\to \frac{\delta_H^2}{4\theta^{1+8H} \sigma_H^2},\\
T \norm{f_T\otimes_{1}g_T}_{\mathfrak{H}^{\otimes 2}}^2&\to 0,\quad
T\norm{g_T\otimes_{1}g_T}_{\mathfrak{H}^{\otimes 2}}^2\to 0.
\end{align*}
\end{corollary}

\begin{lem}\label{2ft2 21}
When $H\in (0,\,\frac12)$,  the speed of convergence $2\norm{f_T}^2_{\mathfrak{H}^{\otimes 2}}\to \big[H\Gamma(2H)\theta^{-2H}\big]^2$ is at least $ T^{2H-1}$ as $T\to \infty$. 
\end{lem}

\begin{proof} Without loss of generality, one can assume that $\theta=1$. One divide the proof into several steps.

Step 1. Similarly to obtain (\ref{ht exp}), one has that
\begin{align}
2\norm{f_T}^2_{\mathfrak{H}^{\otimes 2}}&=\frac{1}{2T \sigma^2_H}\int_{\Rnum^4}\frac{ \partial^2}{\partial t_1 \partial s_2} \big[e^{-\abs{t_1-s_1}-\abs{t_2-s_2}}\mathbf{1}_{[0,T)^4}(t_1,\,s_1,\,t_2,\,s_2)\big] \frac{\partial R_H(t_1,t_2)}{\partial t_2} \frac{\partial R_H(s_1,s_2)}{\partial s_1} \dif \vec{t} \dif \vec{s}\nonumber \\
&=\frac{1}{2T \sigma^2_H}\int_{\Rnum^4}e^{-\abs{t_1-s_1}-\abs{t_2-s_2}}  \mathbf{1}_{[0,T]^2}( s_1,\,t_2 )  \frac{\partial R_H(t_1,t_2)}{\partial t_2} \frac{\partial R_H(s_1,s_2)}{\partial s_1}\nonumber \\
&\times  \Big[\mathbf{1}_{[0,T]^2}( t_1,\,s_2 ) \mathrm{sgn}(t_1-s_1)\mathrm{sgn}(s_2-t_2)-\mathbf{1}_{[0,T] }( t_1  )\mathrm{sgn}(t_1-s_1)(\delta_0(s_2)-\delta_T(s_2))\nonumber \\
&-\mathbf{1}_{[0,T] }( s_2  )\mathrm{sgn}(s_2-t_2)(\delta_0(t_1)-\delta_T(t_1))+( \delta_0(s_2)-\delta_T(s_2))(\delta_0(t_1)-\delta_T(t_1))\Big]\dif \vec{t} \dif \vec{s}\nonumber \\
&:=I_1(T)+I_2(T)+I_3(T)+I_4(T).\label{f2 decomp}
\end{align}

Step 2. Speed of convergence of $I_1(T)\to \big[H\Gamma(2H)\big]^2$. It is proved in Hu, Nualart, and Zhou 2019 that as $T\to \infty$,
\begin{align*}
I_1(T)&=\frac{1}{2T \sigma^2_H}\int_{[0,T]^4}e^{-\abs{t_1-s_1}-\abs{t_2-s_2}} \frac{\partial R_H(t_1,t_2)}{\partial t_2} \frac{\partial R_H(s_1,s_2)}{\partial s_1} \mathrm{sgn}(t_1-s_1)\mathrm{sgn}(s_2-t_2)\dif \vec{t} \dif \vec{s}\nonumber \\
&\to \big[H\Gamma(2H)\big]^2.
\end{align*}
 Since $H\in (0,\frac12)$, the symmetry and the L'Hospital's rule (Taylor 1952)   imply that 
\begin{align}
& \limsup_{T\to \infty} T^{1-2H}\frac{ \sigma^2_H}{H^2} \abs{I_1(T)- \big(H\Gamma(2H)\big)^2}\nonumber\\
&= \limsup_{T\to \infty}  \frac{ \sigma^2_H}{ T^{2H}} \Bigg| \frac12 \int_{[0,T]^4}e^{-\abs{t_1-s_1}-\abs{t_2-s_2}} (t_2^{2H-1}-\mathrm{sgn}(t_2-t_1)\abs{t_2-t_1}^{2H-1}) \nonumber\\
& \times  (s_1^{2H-1}-\mathrm{sgn}(s_1-s_2)\abs{s_1-s_2}^{2H-1}) \mathrm{sgn}(t_1-s_1)\mathrm{sgn}(s_2-t_2)\dif \vec{t} \dif \vec{s} -  \Gamma^2(2H)\sigma^2_HT  \Bigg| \nonumber\\
&\le  \frac{ \sigma^2_H}{2H } \limsup_{T\to \infty}T^{1-2H}  \abs{I_{11}(T) +I_{12}(T)-   \Gamma^2(2H)\sigma^2_H}
 \label{i1 lim speed}
\end{align}
where 
\begin{align}
&\quad I_{11}(T)=\nonumber\\
& \int_{[0,T]^3}e^{t_1-T-\abs{s-t_2}}\mathrm{sgn}(t_2-s) (T^{2H-1}-(T-s)^{2H-1}) (t_2^{2H-1}-\mathrm{sgn}(t_2-t_1)\abs{t_2-t_1}^{2H-1})\dif s\dif\vec{t},\nonumber\\
&\quad I_{12}(T)=\nonumber\\
& \int_{[0,T]^3}e^{s_1-T-\abs{s_2-t}}\mathrm{sgn}( s_2-t) (t^{2H-1}+(T-t)^{2H-1}) (s_1^{2H-1}-\mathrm{sgn}(s_1-s_2)\abs{s_1-s_2}^{2H-1})\dif t\dif\vec{s},\label{i12 exp}
\end{align}please refer to (6.30)-(6.31) of Hu, Nualart, and Zhou 2019.\\
Step 2.1.  An expansion of $I_{11}$.
After dividing the domain of integration $I_{11}$ into two domains according to $s>t_2$ or not and doing a change of variables as in (6.32) of Hu, Nualart, and Zhou 2019,  one has an expansion of $I_{11}$ as follows.
\begin{align}
& I_{11}(T)\nonumber\\
&=\int_{[0,T)^3,\, x\le t} e^{-{u-x}}(T^{2H-1}-t^{2H-1}) \big((T-t+x)^{2H-1}-\mathrm{sgn}(x+u-t)\abs{x+u-t}^{2H-1}\big)\dif u\dif x\dif {t}\nonumber\\
&-\int_{[0,T)^3,\, x\le t} e^{-{u-x}}(T^{2H-1}-(t-x)^{2H-1}) \big((T-t)^{2H-1}-\mathrm{sgn}(u-t)\abs{u-t}^{2H-1}\big)\dif u\dif x\dif {t}\nonumber\\
&=\int_{[0,T]^2} e^{-{u-x}} \sum_{i=1}^4 \varphi_i \dif x\dif u,\label{i11 expansion}
\end{align}
where 
\begin{align*}
\varphi_1(x)&=\int_{x}^T ((t-x)^{2H-1}-T^{2H-1}) \big((T-t)^{2H-1}- (T-t+x)^{2H-1}\big)\dif t,\\
\varphi_2(x)&=\int_{x}^T ((t-x)^{2H-1}-t^{2H-1}) (T-t+x)^{2H-1} \dif t,\\
\varphi_3(x,\,u)&=T^{2H-1}\int_{x}^T\big(\mathrm{sgn}(u-t)\abs{u-t}^{2H-1} -\mathrm{sgn}(x+u-t)\abs{x+u-t}^{2H-1} \big)\dif t,\\
\varphi_4(x,\,u)&=\int_{x}^T\big(t^{2H-1}\mathrm{sgn}(x+u-t)\abs{x+u-t}^{2H-1}-(t-x)^{2H-1} \mathrm{sgn}(u-t)\abs{u-t}^{2H-1}  \big)\dif t.
\end{align*}
Step 2.2.  Speed of convergence $\int_{[0,T]^2} e^{-{u-x}} \sum_{i=1}^3 \abs{\varphi_i} \dif x\dif u \to 0$.
For any fixed $\epsilon\in (0,\, \frac14)$, denote $\mathcal{I}_1=[0,T\epsilon]^2$ and $\mathcal{I}_2=[0,T]^2 \setminus \mathcal{I}_1 $. 
Lemma 13-15 of Hu, Nualart, and Zhou 2019 imply that 
\begin{align*}
\limsup_{T\to \infty} T^{1-2H}\int_{\mathcal{I}_2} e^{-{u-x}} \sum_{i=1}^3 \abs{\varphi_i} \dif x\dif u =0\\
\limsup_{T\to \infty} T^{1-2H}\int_{\mathcal{I}_1} e^{-{u-x}}  ( \abs{\varphi_1} + \abs{\varphi_3} )\dif x\dif u<\infty .
\end{align*}
Moreover, one claims that there exists a constant $c>0$ such that 
\begin{align}
0<\int_{\mathcal{I}_1} e^{-{u-x}} \varphi_2(x) \dif x\dif u
&=(1-e^{-T\epsilon})  \int_{0}^{T\epsilon }e^{-x} \varphi_2(x) \dif x \le  \int_{0}^{T\epsilon }e^{-x} \varphi_2(x) \dif x \le c T^{2H-1}.\label{i1 phi 2}
\end{align}In fact, it is clear that there exists a constant $c_{_{H}}>0$ such that 
\begin{align}\label{z2h ineq}
0<(1-z)^{2H-1}-1<c_{_{H}} z,\qquad \forall z\in (0,\, \frac12].
\end{align}
One  divides the domain of integral of $\varphi_2(x)$ into three parts as follows.
\begin{align} 
 \int_{0}^{T\epsilon }e^{-x} \varphi_2(x) \dif x 
&=\int_{0}^{T\epsilon }e^{-x}  \dif x \Big[ \int_{x}^{2x} +\int_{2x}^{2\epsilon T} +\int_{2\epsilon T}^T \Big]  ((t-x)^{2H-1}-t^{2H-1}) (T-t+x)^{2H-1} \dif t\nonumber \\
&:= J_1+J_2+J_3.\label{j123}
\end{align}The inequality (\ref{z2h ineq}) and the monotonicity of the function $t^{2H-1}$ imply that
\begin{align*}
J_1&\le \big((1-\epsilon)T)^{2H-1} \int_{0}^{T\epsilon }e^{-x}  \dif x \frac{(2-2^{2H})}{2H} x^{2H}<c T^{2H-1},\\
J_2&\le \big((1-2 \epsilon)T)^{2H-1} \int_{0}^{T\epsilon }e^{-x}  \dif x \int_{2x}^{2\epsilon T}c_{_{H}}  \frac{x}{t} t^{2H-1}\dif t<c T^{2H-1},\\
J_3&\le \int_{0}^{T\epsilon }e^{-x} x^{2H-1} \dif x \int_{2\epsilon T}^T c_{_{H}}  \frac{x}{t} t^{2H-1}\dif t<c T^{2H-1}.
\end{align*} Substituting the above three inequalities into the identity (\ref{j123}), one has the inequality (\ref{i1 phi 2}).
Hence, one has that 
\begin{align}
\limsup_{T\to \infty} T^{1-2H}\int_{[0,T]^2} e^{-{u-x}} \sum_{i=1}^3 \abs{\varphi_i} \dif x\dif u 
= \limsup_{T\to \infty} T^{1-2H}\int_{\mathcal{I}_1} e^{-{u-x}} \sum_{i=1}^3 \abs{\varphi_i} \dif x\dif u <\infty.\label{phi 1-3}
\end{align}
Step 2.3. Speed of convergence of $\int_{[0,T]^2} e^{-{u-x}} \varphi_4 \dif x\dif u\to \frac{1}{2}\Gamma^2(2H)\sigma_H^2 $. It is shown (Lemma~14 of Hu, Nualart, and Zhou 2019) that  as $T\to\infty$, 
\begin{align}\label{phi4 integ}
\int_{[0,T]^2} e^{-{u-x}} \varphi_4 \dif x\dif u\to \frac{1}{2}\Gamma^2(2H)\sigma_H^2, 
\end{align}
and the integral can be decomposed as follows:
\begin{align*}
\int_{[0,T]^2} e^{-{u-x}} \varphi_4 \dif x\dif u:=L_1(T)-L_2(T)+L_3(T),
\end{align*}
where 
\begin{align*}
L_1(T)&=\int_{[0,T]^2} e^{-{u-x}} \dif x\dif u\int_{x}^{x+u} t^{2H-1}(x+u-t)^{2H-1}\dif t ,\\
L_2(T)&=\int_{[0,T]^2,\,x< u} e^{-{u-x}} \dif x\dif u\int_{x}^{u} (t-x)^{2H-1}(u-t)^{2H-1}\dif t ,\\
L_3(T)&=\int_{[0,T]^2} e^{-{u-x}} \dif x\dif u\Big[\int_{x \vee u}^{T} (t-x)^{2H-1}(t-u)^{2H-1} \dif t - \int_{x+u}^{T} t^{2H-1}(t-x-u)^{2H-1}\dif t\Big],
\end{align*} 
Hence, one has that 
\begin{align}\label{varphi 4}
\limsup_{T\to \infty} T^{3-4H}\abs{\int_{[0,T]^2} e^{-{u-x}} \varphi_4 \dif x\dif u- \frac{1}{2}\Gamma^2(2H)\sigma_H^2}=\sum_{i=1}^3 \limsup_{T\to \infty} T^{3-4H}\abs{L_i(T)- L_i(\infty) }.
\end{align}
It is clear that 
\begin{align*}
0&< L_1(\infty)-L_1(T)\\
&=\int_{\Rnum_{+}^2\setminus [0,T]^2} e^{-{u-x}} \dif x\dif u\int_{x}^{x+u} t^{2H-1}(x+u-t)^{2H-1}\dif t ,\\
&=\int_{\Rnum_{+}^2\setminus [0,T]^2} e^{-{u-x}} \dif x\dif u\int_{0}^{u} (x+s)^{2H-1}(u-s)^{2H-1}\dif s,\\ 
&=\Big[\int_0^T\dif x\int_0^T \dif s \int_{T}^{\infty}\dif u+ \int_0^T\dif x\int_T^{\infty} \dif s \int_{s}^{\infty}\dif u + \int_T^{\infty} \dif x \int_0^{\infty}\dif s\int_{s}^{\infty}\dif u\Big]\\
&\times  e^{-{u-x}}   (x+s)^{2H-1}(u-s)^{2H-1}\\
&:=L_{11}+L_{12}+L_{13},
\end{align*}where
\begin{align*}
L_{11}&<\int_0^T\dif x\int_0^T \dif s \int_{T}^{\infty}e^{-{u-x}}   (x+s)^{2H-1}(u-T)^{2H-1}\dif u< c T e^{-T},\\
L_{12}&<\int_0^T\dif x\int_T^{\infty} \dif s  \int_{0}^{\infty} e^{-v-s-x}   s^{2H-1}v ^{2H-1}\dif v< c T^{2H-1} e^{-T},\\
L_{13}&<\int_T^{\infty} \dif x \int_0^{\infty}\dif s\int_{0}^{\infty}e^{-v-s-x}   s^{2H-1}v ^{2H-1} \dif v< c e^{-T},
\end{align*} which imply that $L_1(T)\to L_1(\infty)$ with an exponential rate as $T\to \infty$.
It is obvious that  $L_2(T)\to L_2(\infty)$ also with an exponential rate as $T\to \infty$. In fact,
\begin{align*}
0&< L_2(\infty)-L_2(T)\\
&=\int_{\Rnum_{+}^2\setminus [0,T]^2,\,x< u} e^{-{u-x}} \dif x\dif u\int_{x}^{u} (t-x)^{2H-1}(u-t)^{2H-1}\dif t\\
&=B(2H,\,2H)\int_T^{\infty}\dif x\int_{x}^{\infty} e^{-{u-x}}  (u-x)^{4H-1}  \dif u \\
&=B(2H,\,2H)\Gamma(2H)\int_T^{\infty}e^{-2x} \dif x< c e^{-2T}.
\end{align*}
Since $(t-x)(t-u)\ge t(t-x-u)$ for $x,u>0$, the symmetry and the monotonicity of the function $t^{2H-1}$ imply that 
\begin{align*}
&\quad \frac12 \abs{ L_3(\infty)-L_3(T)}\\
&=\Big|\int_{0<u<x<t,\,t\le x+u,\,t>T} e^{-{u-x}} (t-x)^{2H-1}(t-u)^{2H-1} \dif t  \dif x\dif u  \\
&-\big(-\int_{0<u<x<t,\,t> x+u,\,t>T} e^{-{u-x}} (t-x)^{2H-1}(t-u)^{2H-1} \dif t  \dif x\dif u  \\
&+ \int_{0<u<x,\,t>x+u,\,t>T} e^{-{u-x}}  t^{2H-1}(t-x-u )^{2H-1}\dif t  \dif x\dif u\big)\Big|\\
&:=\abs{K_1(T)-K_2(T)}< K_1(T)+K_2(T)  .
\end{align*}
It is clear that 
\begin{align*}
K_1(T)&= \int_{0<u<x<t \le x+u,\,t>T} e^{-{u-x}} (t-x)^{2H-1}(t-u)^{2H-1} \dif t  \dif x\dif u\\
&=\int_{T}^{\infty} \dif t \int_{\frac{t}{2}}^{t}\dif x\int_{t-x}^{x}\dif u\,\,e^{-{u-x}} (t-x)^{2H-1}(t-u)^{2H-1}\\
&<\int_{T}^{\infty}e^{-t} \dif t \int_{\frac{t}{2}}^{t}(t-x)^{2H-1}\dif x\int_{t-x}^{x}\,\, (t-u)^{2H-1}\dif u\\
&<c T^{4H}e^{-T},
\end{align*}
and  
\begin{align*}
K_2(T)&=\int_{T}^{\infty} \dif t \int_0^{\frac{t}{2}}\dif u\int_{u}^{t-u} \dif x\,\,e^{-{u-x}}  \big[t^{2H-1}(t-x-u )^{2H-1}-(t-x)^{2H-1}(t-u)^{2H-1} \big]  \\
&=\int_{T}^{\infty} e^{-t}  t^{2H-1} \dif t \int_0^{\frac{t}{2}}\dif u\int_0^{t-2u} e^{z}z^{2H-1}  \dif z-\int_{T}^{\infty} e^{-2t} \dif t \int_{\frac{t}{2}}^{t} e^{y}y^{2H-1}\dif y\int_{t-y}^{y} e^{z}z^{2H-1}  \dif z,
\end{align*}where the last equality is by the change of variables $t-x-u=z$ and $t-x=z,\,t-u=y$ respectively.
Then the L'Hospital's rule implies that 
\begin{align*}
&\lim_{T\to\infty}\frac{2(3-4H)}{T^{4H-3}}K_2(T)\\
&=\lim_{T\to\infty}\frac{2}{ e^{2T}T^{4H-4}} \Big[e^T T^{2H-1} \int_0^{\frac{T}{2}}\dif u\int_0^{T-2u} e^{z}z^{2H-1}  \dif z -\int_{\frac{T}{2}}^{T} e^{y}y^{2H-1}\dif y\int_{T-y}^{y} e^{z}z^{2H-1}  \dif z\Big]\\
&=\lim_{T\to\infty} \frac{e^T T^{2H-1}}{ e^{2T}T^{4H-4}} \Big[(1+\frac{2H-1}{T})\int_0^{\frac{T}{2}}\dif u\int_0^{T-2u} e^{z}z^{2H-1}  \dif z-\frac12 \int_0^T e^{z}z^{2H-1}  \dif z+\frac12 B(2H,2H)T^{2H} \Big]\\
&=\lim_{T\to\infty} \frac{1}{ e^{T} T^{2H-2}}\Big[(T+2H-1)  \int_0^{\frac{T}{2}}\dif u\int_0^{T-2u} e^{z}z^{2H-1}  \dif z-\frac{T}{2} \int_0^T e^{z}z^{2H-1}  \dif z  \Big]\\
&=\lim_{T\to\infty} \frac{1}{ e^{T} T^{2H-2}}\big[\int_0^T e^z z^{2H-1}\dif z -e^T T^{2H-1} \big]\\
&=  (1-2H).
\end{align*}
Hence it follows from (\ref{varphi 4}) that one has that  
\begin{align}\label{phi 4}
\limsup_{T\to \infty} T^{3-4H}\abs{\int_{[0,T]^2} e^{-{u-x}} \varphi_4 \dif x\dif u- \frac{1}{2}\Gamma^2(2H)\sigma_H^2}{ \le}\frac{1-2H}{3-4H}< \infty.
\end{align}

Combining (\ref{phi 4}) and (\ref{phi 1-3}) with (\ref{i11 expansion}), one has that 
\begin{align}
&\limsup_{T\to \infty}T^{1-2H}  \abs{I_{11}(T) - \frac12 \Gamma^2(2H)\sigma^2_H}\nonumber \\
&\le \limsup_{T\to \infty} T^{1-2H}\int_{[0,T]^2} e^{-{u-x}} \sum_{i=1}^3 \abs{\varphi_i} \dif x\dif u + \limsup_{T\to \infty} T^{1-2H}\abs{\int_{[0,T]^2} e^{-{u-x}} \varphi_4 \dif x\dif u- \frac{1}{2}\Gamma^2(2H)\sigma_H^2}\nonumber \\
&= \limsup_{T\to \infty} T^{1-2H}\int_{[0,T]^2} e^{-{u-x}} \sum_{i=1}^3 \abs{\varphi_i} \dif x\dif u< \infty. \label{i11 speed}
\end{align}
Step 2.4. Speed of convergence $I_{12}(T) \to \frac12 \Gamma^2(2H)\sigma^2_H$. It follows from (6.36) of Hu, Nualart, and Zhou 2019 that
\begin{align}
I_{12}(T)&=\int_{[0,T]^2} e^{-{u-x}} ( \varphi_4+ \varphi_5)\dif x\dif u,\label{i12 expansion}
\end{align}where 
\begin{align*}
\varphi_5=\int_{x}^T\big(\mathrm{sgn}(x+u-t)\abs{x+u-t}^{2H-1}(T-t)^{2H-1}- \mathrm{sgn}(u-t)\abs{u-t}^{2H-1}(T-t+x)^{2H-1}  \big)\dif t. 
\end{align*}
Similar to Step 2.2, it follows from Lemma 13-15 of Hu, Nualart, and Zhou 2019 that  as $T\to \infty$,
\begin{align}\label{phi 5}
\limsup_{T\to \infty}T^{1-2H} \int_{[0,T]^2} e^{-{u-x}} \abs{ \varphi_5}\dif x\dif u=\limsup_{T\to \infty}\int_{\mathcal{I}_1} e^{-{u-x}}\abs{ \varphi_5} \dif x\dif u<\infty.
\end{align}
Combining (\ref{phi 4}) and (\ref{phi 5}) with (\ref{i12 expansion}), one has that 
\begin{align}\label{i12 speed}
\limsup_{T\to \infty}T^{1-2H}  \abs{I_{12}(T) - \frac12 \Gamma^2(2H)\sigma^2_H}\le \limsup_{T\to \infty}T^{1-2H} \int_{[0,T]^2} e^{-{u-x}} \abs{ \varphi_5}\dif x\dif u<\infty.
\end{align}

Combining (\ref{i11 speed}) and (\ref{i12 speed}) with (\ref{i1 lim speed}), one has that
\begin{align*}
 &\limsup_{T\to \infty} T^{1-2H} \abs{ I_1(T)- \big(H\Gamma(2H)\big)^2}\\
 &\le\limsup_{T\to \infty}T^{1-2H}  \abs{I_{11}(T) - \frac12 \Gamma^2(2H)\sigma^2_H}+ \limsup_{T\to \infty}T^{1-2H}  \abs{I_{12}(T) - \frac12 \Gamma^2(2H)\sigma^2_H}\\
 &<\infty.
\end{align*}

That is to say, the speed of convergence $I_1\to \big[H\Gamma(2H)\big]^2$ is at least $T^{2H-1}$.  

Step 3. Speeds of convergences $I_2(T),\,I_3(T)\to 0$. One has that 
\begin{align*}
& \quad \frac{2\sigma^2_H}{H^2} I_2(T) \nonumber\\
&=\frac{1}{T e^T}\int_{[0,T)^3}e^{t_2-\abs{t_1-s}}\mathrm{sgn}(s-t_1) (s^{2H-1}+(T-s)^{2H-1})(t_2^{2H-1}-\mathrm{sgn}(t_2-t_1)\abs{t_2-t_1}^{2H-1})\dif s\dif\vec{t}. \label{i2 exp}
\end{align*}
Comparing  it with the identity (\ref{i12 exp}), one has that 
\begin{align*}
 \quad \frac{2\sigma^2_H}{H^2} I_2(T)=\frac{1}{T}I_{12}(T).
\end{align*}Hence, the equalities (\ref{phi4 integ}) and (\ref{i12 expansion}) imply that there exists a constant $c>0$ such that for $T$ large enough,
\begin{align*}
\abs{I_2(T)}\le \frac{c}{T}.
\end{align*}By the symmetry, one has that $I_2=I_3$.

Step 4. Speed of convergence $I_4(T)\to 0$. It is clear that there exists a constant $c>0$ such that as $T$ large enough,
\begin{align*}
I_4(T)=\frac{H^2}{2\sigma^2_H T e^{2T}}\Big[\int_0^{T}  e^{ t }\big(t^{2H-1}+(T-t)^{2H-1}\big)  \dif {t}\Big]^2<\frac{c}{T}.
\end{align*}

Finally, substituting speeds of convergences obtained at Step 2-4 to (\ref{f2 decomp}), one has the desired conclusion.
\end{proof}

 After the above three lemmas are shown,  proof of Theorem~\ref{main thm} is almost the same as that of the case of $H\in [\frac12,\,\frac34)$, please refer to Chen, Kuang, and Li 2018. But for the reader's convenience, one still writes it here. The only difference is the upper bound in the inequality (\ref{lastt}) given below.  \\[0.6mm]
\noindent{\it Proof of Theorem~\ref{main thm}.\,} 
It follows from Theorem~\ref{kp},\, Lemma~\ref{bt exponential}  and Eq.(\ref{ratio 1})-(\ref{bt bt}) that  there exists a constant $C_{\theta, H}$ such that for $T$ large enough,
\begin{align*}
&\sup_{z\in \Rnum}\abs{P(\sqrt{\frac{T}{\theta \sigma^2_H}} (\hat{\theta}_T-\theta )\le z)-P(Z\le z)}\le  \\
& C_{\theta, H}\times\max\set{\abs{b_T^2-2\norm{f_T}^2},\,\norm{f_T\otimes_1 f_T},\, \norm{f_T\otimes_1 g_T},\,\abs{\innp{f_T,\,g_T}},\,\norm{g_T}^2,\,\norm{g_T\otimes_1 g_T}}.\end{align*}
Denote $a=H\Gamma(2H)\theta^{-2H}$.
Lemma~\ref{bt exponential} and Lemma~\ref{2ft2 21} imply that  there exists a constant $c$ such that  for $T$ large enough,
\begin{equation}\label{lastt}
\abs{b_T^2-2\norm{f_T}^2}\le \abs{b_T^2- a^2}+\abs{2\norm{f_T}^2-a^2}\le c\times \frac{1}{T^{1-2H}}.
\end{equation}
Corollary~\ref{gt gt} implies that  there exists a constant $c$ such that  for $T$ large enough,
\begin{align*}
\norm{f_T\otimes_1 g_T},\,\abs{\innp{f_T,\,g_T}},\,\norm{g_T\otimes_1 g_T}  \le c\times \frac{1}{\sqrt{T}},\qquad
\norm{g_T}^2\le c\times \frac{1}{{T}}.
\end{align*}
Combining (\ref{zhou ineq}) with the above inequalities, one obtains that  (\ref{hle 34}) holds.  
{\hfill\large{$\Box$}}
\vskip 0.2cm {\small {\bf  Acknowledgements}:
We would like to gratefully thank the referee for very valuable suggestions which lead to the improvement of the new version. 
 Y. Chen is supported by NSFC (No.11871079).
}



\end{document}